\documentclass[11pt]{amsart}
\usepackage[a4paper,margin=2.4cm]{geometry}
\usepackage{mathtools,amsmath,amsthm,amssymb,amscd,amsfonts,color}
\usepackage{dsfont}
\usepackage{kbordermatrix}
\usepackage{multirow}
\usepackage{tikz}
\usepackage{booktabs}
\usepackage{diagbox}
\usepackage{faktor}
\usepackage{thmtools,thm-restate}
\usepackage{hyperref}
\usepackage{bm}
\usepackage{pst-node}%
\usepackage{hhline}

\theoremstyle{plain}
\newtheorem{theorem}{Theorem}[section]
\newtheorem{proposition}[theorem]{Proposition}
\newtheorem{lemma}[theorem]{Lemma}
\newtheorem{corollary}[theorem]{Corollary}
\numberwithin{equation}{section}

\theoremstyle{definition}

\newtheorem{example}[theorem]{Example}
\newtheorem{problem}[theorem]{Problem}
\newtheorem{remark}[theorem]{Remark}
\newtheorem{algorithm}[theorem]{Algorithm}

\newcommand{\C}{\mathbb{C}}
\newcommand{\R}{\mathbb{R}}
\newcommand{\Z}{\mathbb{Z}}

\newcommand{\cP}{\mathcal{P}}

\newcommand{\cJ}{\mathcal{J}}
\newcommand{\cF}{\mathcal{F}}
\newcommand{\cR}{\mathcal{R}}

\newcommand{\cX}{\mathcal{X}}

\newcommand{\cM}{\mathcal{M}}
\newcommand{\cZ}{\mathcal{Z}}
\newcommand{\cA}{\mathcal{A}}

\DeclareMathOperator{\lk}{Lk}
\DeclareMathOperator{\wed}{Wed}
\DeclareMathOperator{\proj}{Proj}
\DeclareMathOperator{\susp}{\Sigma}
\DeclareMathOperator{\Hom}{Hom}

\title[Toric wedge induction and toric lifting property]{Toric wedge induction and toric lifting property for piecewise linear spheres with a few vertices}

\author{Suyoung Choi}
\address{Department of mathematics, Ajou University, 206, World cup-ro, Yeongtong-gu, Suwon 16499,  Republic of Korea}
\email{schoi@ajou.ac.kr}
\author{Hyeontae Jang}
\address{Department of mathematics, Ajou University, 206, World cup-ro, Yeongtong-gu, Suwon 16499, Republic of Korea}
\email{a24325@ajou.ac.kr}
\author{Mathieu Vall\'ee}
\address{Université Sorbonne Paris Nord, LIPN, CNRS UMR 7030, F-93430, Villetaneuse, France}
\email{vallee@lipn.fr}

\date{\today}


\thanks{This work was supported by the National Research Foundation of Korea Grant funded by the Korean Government (NRF-2021R1A6A1A10044950).}

\begin{document}
\begin{abstract}
    Let $K$ be an $(n-1)$-dimensional piecewise linear sphere on $[m]$, where $m\leq n+4$.
    There are a canonical action of $m$-dimensional torus $T^m$ on the moment-angle complex $\cZ_K$, and a canonical action of $\Z_2^m$ on the real moment-angle complex $\R\cZ_K$, where $\Z_2$ is the additive group with two elements.
    We prove that any subgroup of $\Z_2^m$ acting freely on $\R\cZ_K$ is induced by a subtorus of $T^m$ acting freely on $\cZ_K$.
    The proof primarily utilizes a suitably modified method of toric wedge induction and the combinatorial structure of a specific binary matroid of rank~$4$.
\end{abstract}
\maketitle

\tableofcontents
\section{Introduction}
Let $K$ be a simplicial complex on the set $[m]=\{1, \ldots, m\}$.
We define the \emph{polyhedral product} $ (\underline{X}, \underline{Y})^K $ of $K$ with respect to a pair $(X, Y)$ of topological spaces as follows:
$$
    (\underline{X}, \underline{Y})^K \coloneqq \bigcup_{\sigma \in K} \left\{ (x_1, \ldots, x_m) \in X^m \mid x_i \in Y \text{ when } i \notin \sigma \right\}.
$$
Here, $D^d$ represents the $d$-dimensional disk, defined as $D^d = \{\mathbf{x} \in \R^d \mid \|\mathbf{x}\| \leq 1\}$, and $S^{d-1}$ denotes its boundary sphere of dimension $d-1$.
The \emph{moment-angle complex} $\cZ_K$ of $K$ is then defined as $(\underline{D^2}, \underline{S^1})^K$, and the \emph{real moment-angle complex} $\R\cZ_K$ of $K$ is $(\underline{D^1}, \underline{S^0})^K$.
We observe that the $T^1$-action on the pair $(D^2, S^1)$ leads to the canonical action of the $m$-dimensional torus $T^m = (S^1)^m$ on $\cZ_K$.
Additionally, there is an $S^0$-action on the pair $(D^1, S^0)$.
For clarity and consistency in our terminology throughout this paper, we treat $S^0$ as the additive group $\Z_2=\Z/2\Z$ with two elements $\{ 0, 1\}$.
This, then, yields the canonical $\Z_2^m$-action on $\R\cZ_K$.

It is noteworthy that when an $r$-dimensional subtorus $H$ of $T^m$ acts freely on $\cZ_K$, the resulting quotient space $\cZ_K /H$ admits a well-behaved torus action $T^m/H \cong T^{m-r}$ with an orbit space that exhibits a reverse face structure isomorphic to $K$.
Such spaces are commonly referred to as \emph{toric spaces} or \emph{(partial) quotients}, and are fundamental in the study of \emph{toric topology}.
Consequently, understanding which subtori $H$ of $T^m$ can act freely on $\cZ_K$ is of significant importance.
The Buchstaber number $s(K)$ is the maximal integer $r$ for which there exists a subtorus of rank $r$ acting freely on $\cZ_K$.
Similarly, taking a subgroup $H$ of $\Z_2^m$ freely acting on $\R\cZ_K$ yields the quotient space $\R\cZ_K / H$ which is referred to as a \emph{real toric space} or a \emph{real (partial) quotient}.
The real Buchstaber number $s_\R(K)$ is similarly defined by the existence of a subgroup acting freely on $\R\cZ_K$.
The determination of (real) Buchstaber numbers is challenging.
We refer to the following publications for details:
\cite{BP2002}, \cite{Fukukawa2011}, \cite{Erokhovets2014}, \cite{Ayzenberg2016}, and \cite{Shen2023}.

It is known that the real moment-angle complex $\R\cZ_K$ is the fixed point set by the involution on $\cZ_K$ induced by the complex conjugation on $D^2 \subset \C$.
This implies that a $T^m$-action on $\cZ_K$ induces a $\Z_2^m$-action on $\R\cZ_K$, and then $d$-dimensional subtorus of $T^m$ acting freely on $\cZ_K$ induces a rank $d$ subgroup of $\Z_2^m$ acting freely on $\R\cZ_K$.
Thus, we obtain the inequality $s(K) \leq s_\R(K)$, and Ayzenberg \cite{Ayzenberg2011} noted that the equality does not generally hold; specifically, there exists a simplicial complex whose real Buchstaber number is strictly bigger than its Buchstaber number.

From now on, we zero in on the case when $K$ is a PL~sphere, since in this case, all toric spaces over \( K \) are PL~manifolds  \cite{Cai2017}.
If $K$ is $(n-1)$-dimensional, we have the inequalities $s(K)\leq s_\R(K) \leq m-n$.
Given the condition $s(K) = m-n$, which is a special case often encountered in various fields of mathematics, the manifold $\cZ_K / H$ for a maximal subtorus $H \subset T^m$ freely acting on $\cZ_K$ is termed a \emph{topological toric manifold} \cite{Ishida-Fukukawa-Masuda2013} when $K$ is star-shaped.
If $K$ is polytopal, the manifold is referred to as a \emph{quasitoric manifold} \cite{Davis-Januszkiewicz1991}\footnote{A quasitoric manifold was originally called a toric manifold in \cite{Davis-Januszkiewicz1991}, and was renamed in \cite{BP2002} to avoid confusion with a smooth compact toric variety.}.
Similarly, given the condition $s_\R(K)=m-n$, the manifold $\R\cZ_K / H$ for a maximal subgroup $H$ freely acting on $\R\cZ_K$ is called a \emph{real topological toric manifold} when $K$ is star-shaped, and it is called a \emph{small cover} when $K$ is polytopal.
These are real analogs of topological toric and quasitoric manifolds, respectively.

In the class of PL~spheres, no examples have been known where  $s(K) < s_\R(K)$.
In light of this observation, one may ask whether $s(K)=s_\R(K)$ for a PL~sphere $K$, and the following stronger question can be considered.

\begin{problem} \label{problem}
    Let $K$ be a PL~sphere on $[m]$.
    Given a subgroup of $\Z_2^m$ acting freely on $\R\cZ_K$, is this action induced by a subtorus of $T^m$ freely acting on $\cZ_K$?
\end{problem}

In particular, when $s_\R(K) = m-n$, Problem~\ref{problem} is equivalent to the (toric) \emph{lifting problem} (Problem~\ref{equiv_problem}).
In other words, this asks whether every small cover (or real topological toric manifold) is induced from some quasitoric manifold (or topological toric manifold, respectively).
The lifting problem was initially proposed by Zhi L\"{u} at the toric topology conference held in Osaka in 2011, as documented in \cite{Choi-Park2016}, and remains an open problem in toric topology, attracting considerable research attention.
However, significant advances in resolving this problem have been elusive.
This paper aims to make a contribution by providing meaningful results to the lifting problem, and more broadly to Problem~\ref{problem}, in the case $m \leq n+4$.

\begin{restatable}{theorem}{main} \label{thm:main}
     Let $K$ be an $(n-1)$-dimensional PL~sphere with $m \leq n+4$ vertices.
     Then, any subgroup of $\Z_2^m$ freely acting on $\R\cZ_K$ is induced by a subtorus of $T^m$ freely acting on $\cZ_K$.
\end{restatable}

Let $H$ be a subgroup of $\Z_2^m$  of rank $r$, freely acting on $\R\cZ_K$.
Then, $0 \leq r \leq \min(s_R(K),m-n) \leq 4$.
To address this question, we categorize our approach into two distinct cases: the case where $r \leq 3$ and the case where $r=s_\R(K)= m-n=4$.
In Section~\ref{sec:m-l<=3}, we give a positive answer to Problem~\ref{problem} under the condition that for $r \leq 3$ without the necessity of $m=n+4$.
It demonstrates that Theorem~\ref{thm:main} holds for the case where $m \leq n+3$, or the case  $m=n+4$ and $r \leq 3$.

Subsequent sections will focus on the case $r = 4$ under the specific condition $m-n=s_\R(K)=4$.
Here, we establish the theorem for this case employing a method we call \emph{toric wedge induction}.
This method was firstly introduced by Choi and Park in \cite{Choi-Park2016}, and it can be effectively used to demonstrate properties of toric varieties for certain Picard numbers. 
In our proof, we will introduce a more powerful version of toric wedge induction. 
Additionally, the combinatorial structure of a binary matroid of rank~$4$ will be crucially used for the basis step of the induction.

\section{The case $r \leq 3$} \label{sec:m-l<=3}
Let $K$ be an $(n-1)$-dimensional simplicial complex on $[m]=\{1, 2, \ldots, m\}$, $H \subset T^m$ a subtorus of dimension $r \leq m-n$.
After choosing a basis, it can be written as
\begin{equation}\label{free_action}
    H = \{(e^{2 \pi i (s_{11} \phi_1 + \dots + s_{1r} \phi_r)}, \ldots, e^{2 \pi i (s_{m1} \phi_1 + \dots + s_{mr} \phi_r)}) \in T^m \mid \phi_j \in \R, j=1, ..., r\},
\end{equation}
where $s_{ij} \in \Z$.
We define an $m \times r$ integer matrix $S=(s_{ij})$.
Additionally, the $(m-n) \times r$ matrix $S_{\hat{i}_1, \ldots, \hat{i}_n}$ is defined as the submatrix of $S$ obtained by excluding the rows corresponding to entries $i_j$ for $j=1, \ldots,n$.
The following proposition was proved for polytopal simplicial complexes in \cite{BP2002}, but it can be also proved by a similar argument for general ones.

\begin{proposition} \label{proposition:free action}
    Let $K$ be a simplicial complex. Then the subtorus \eqref{free_action} acts freely on $\cZ_K$ if and only if for any facet $\{i_1, \ldots, i_n\}$ of $K$, the matrix $S_{\hat{i}_1, \ldots, \hat{i}_n}$ defined above gives a monomorphism $\Z^r \longrightarrow \Z^{m-n}$ to a direct summand.
\end{proposition}

The latter condition is equivalent to $S_{\hat{i}_1, \ldots, \hat{i}_n}$ having an $r \times r$ submatrix whose determinant is $\pm 1$.
A similar argument holds for $\Z_2^m$-action on $\R\cZ_K$.
Then the mod~$2$ reduction of the matrix $S$ representing a freely acting subtorus of $T^m$ on $\cZ_K$ represents a freely acting subgroup of $\Z_2^m$ on $\R \cZ_K$.

\begin{theorem} \label{thm: r3}
    Let $K$ be an $(n-1)$-dimensional simplicial complex on $[m]$, and $r \leq 3$ a non-negative integer.
    Then any rank $r$ subgroup of $\Z_2^m$ acting freely on $\R\mathcal{Z}_K$ is induced by an $r$-dimensional subtorus of $T^m$ freely acting on $\cZ_K$.
\end{theorem}

\begin{proof}
    Assume an $m \times r$ matrix $S$ over $\Z_2$ represents the freely acting subgroup of $\Z_2^m$ on $\R\cZ_K$.
    Define the $m \times r$ matrix $\tilde{S}$ over $\Z$ all of whose entries $\tilde{s}_{ij}$ are in $\{0, 1\}$ and such that $S \equiv \tilde{S}$ mod~$2$.
    For any facet $\{i_1, \ldots, i_n \}$ of $K$, $S_{\hat{i}_1, \ldots, \hat{i}_n}$ has an $r \times r$ submatrix $R$ whose determinant is $1 \in \Z_2$.
    Then the corresponding submatrix $\tilde{R}$ of $\tilde{S}_{\hat{i}_1, \ldots, \hat{i}_n}$ has an odd determinant.
    Since the absolute value of every square $\{0,1\}$-matrix of size $r \leq 3$ is less than $r$, the determinant of $\tilde{R}$ is indeed $\pm 1 \in \Z$.
    Hence $\tilde{S}$ defines an $r$-dimensional subtorus of $T^m$ acting freely on $\cZ_K$.
\end{proof}

\section{The case $r =4$ : Preliminaries}

\subsection{Characteristic and dual characteristic maps}
Let $A$ be an $n \times m$ matrix over $\Z$ for positive integers $n \leq m$, and $I$ an $n$-subset of $[m]$.
Let $A_I$ denote  the submatrix of $A$ formed by selecting columns indexed with $i \in I$, and $A^I$ the submatrix of $A$ formed by selecting rows indexed with $i \in I$.
Furthermore, $\overline{A}$ represents a matrix whose columns form a basis of the kernel of $A$.
Note that $\overline{A}$ depends on the choice of a basis of the kernel of $A$.
We introduce one important proposition, known as the \emph{linear Gale duality}:

\begin{proposition} \label{proposition: linear Gale duality}
    Let $A$ be an $n \times m$ matrix over $\Z$ for positive integers $n \leq m$.
    For any $n$-subset $I$ of $[m]$, $\det(A_I) = \pm 1$ if and only if $\det(\overline{A}^{I^c}) = \pm 1$
\end{proposition}

Let $K$ be an $(n-1)$-dimensional PL~sphere on $[m]$, and $H$ an \mbox{$r$-dimensional} subtorus of $T^m$ freely acting on $\cZ_K$ as in \eqref{free_action}.
If $r = m-n$, then $H$ is completely described as follows. 
Let us consider a map $\lambda \colon [m] \to \Z^n$, called a \emph{characteristic map} over $K$, such that $\{ \lambda(i_1), \ldots, \lambda(i_k) \}$ is a unimodular set for any simplex $\{i_1, \ldots, i_k\}$ in $K$.
For convenience, we often represent this map by an $n \times m$ matrix
$\lambda=
\begin{bmatrix}
    \lambda(1) & \cdots & \lambda(m)
\end{bmatrix}
$
with elements in $\Z$.
This matrix can be interpreted as a linear map $\Z^m \to \Z^n$, and concurrently, as an element in $\Hom(T^m, T^n)$.
In addition, we call $\overline{\lambda} \colon [m] \to \Z^{m-n}$ a \emph{dual characteristic map} over $K$.
Similarly, we define \emph{mod~$2$ characteristic maps} $\lambda^\R \colon [m] \to \Z_2^n$ over $K$ and \emph{mod~$2$ dual characteristic maps} $\overline{\lambda}^\R \colon [m] \to \Z_2^{m-n}$ over $K$.
In particular, an injective mod~$2$ dual characteristic map is simply called an \emph{IDCM}.

For an $n \times m$ matrix $\lambda$, $\overline{\lambda}$ defines a subtorus $H$ of $T^m$ similar to that described in $\eqref{free_action}$.
By Propositions~\ref{proposition:free action} and ~\ref{proposition: linear Gale duality}, $\lambda$ is a characteristic map over $K$ if and only if the corresponding subtorus $H$ of $\overline{\lambda}$ acts on $\cZ_K$ freely.

When considering the toric space $\cZ_K / H$, the kernel of $\lambda$ itself is essential whereas the choice of a basis of the kernel is not important.
In this context, we consider the concepts of \emph{Davis-Januszkiewicz equivalence}, or simply \emph{D-J equivalence}, for characteristic maps and dual characteristic maps.
Two characteristic maps are said to be D-J equivalent, if one is obtained by row operations from the other.
Two dual characteristic maps are said to be D-J equivalent if one is obtained by column operations from the other.
This also removes the ambiguity arising from the definition of $\overline{\lambda}$.

Observe that the mod~$2$ reduction of a characteristic map $\lambda$ over $K$ is a mod~$2$ characteristic map over $K$.
Conversely, given a mod~$2$ characteristic map $\lambda^\R \colon [m] \to \Z_2^n$ over $K$ and a characteristic map $\lambda \colon [m] \to \Z^n$ over $K$, if $\lambda^\R$ coincides with the composition of $\lambda$ and the modulo $2$ reduction map $\Z^n \to \Z_2^n$, then $\lambda$ is called a \emph{lift} of $\lambda^\R$:

\begin{center}
\begin{tikzpicture}
  \node (A) at (0,2) {$\Z^n$};
  \node (B) at (-2,0) {$[m]$};
  \node (C) at (2,0) {$\Z_2^n$.};

  \draw[->] (B) --  node[below] {$\lambda^\R$} (C);
  \draw[->, dashed] (B) --  node[left] {$~^\exists \lambda \;$} (A);
  \draw[->] (A) -- node[right] {$\mod 2$} (C) ;
\end{tikzpicture}  
\end{center}

Moreover, it is called the \emph{$\{0,1\}$-lift} of $\lambda^\R$ when $\lambda$ sends $[m]$ to $\{0,1\}$-vectors.
Similarly, it is called a \emph{$\{0,\pm 1\}$-lift} when it sends $[m]$ to $\{0,\pm 1\}$-vectors.
Note that the number of $\{0,\pm1 \}$-lifts of a given mod~$2$ characteristic map is finite.

\begin{example} \label{example: general Bott}
    Let $K$ be the join $\partial \Delta^{n_1} * \partial \Delta^{n_2} * \cdots * \partial \Delta^{n_p}$ of the boundaries of $p$ simplices.
    We denote its set of vertices as $$\{i_j \mid 1 \leq i \leq p, 1 \leq j \leq n_i+1 \},$$ where $i_1$, $i_2$, $\ldots$, $i_{n_i+1}$ comes from the vertices of $\partial \Delta^{n_i}$.
    Let $F^i_j = \{i_1, i_2, \ldots, i_{n_j+1}\} \setminus \{i_j\}$.
    The set of facets of $K$ is $$\{\cup_{i=1} ^{p} F^i_{j_i} \mid 1 \leq j_i \leq {n_i+1} \}. $$
    
    By \cite{Choi-Masuda-Suh2010}, up to D-J equivalence and vertex relabeling, a mod~$2$ characteristic map over $K$ is of the form

    $$\lambda^\R = \begin{bmatrix}
                        \lambda_1^\R              &               &        &                    &                  \\
                        \ast                      & \lambda_2^\R  &        &                    &                  \\
                        \multirow{2}{*}{$\vdots$} & \ast          & \ddots &                    &                   \\
                                                  & \vdots        &        & \lambda_{p-1}^\R &                  \\
                        \ast                      & \ast          &        & \ast               & \lambda_{p}^\R \\
                    \end{bmatrix},$$
                     where $\lambda^\R_i$ is a mod~$2$ characteristic map over $\Delta^{n_i}$ and the empty spaces display zeros.
    Up to D-J equivalence and vertex relabeling, $$\lambda^\R_i = \kbordermatrix{ & i_1 & \cdots & i_{n_i} & i_{n_i+1} \\
                    & \multicolumn{3}{c}{$I_{n_i}$} & \begin{array}{c} 1\\\vdots\\1\end{array}}.$$
    Let $\widetilde{\lambda^\R}$ be the $\{0, 1\}$-matrix over $\Z$ such that $\widetilde{\lambda^\R} \equiv \lambda^\R$ mod~$2$. 
    We denote by $\widetilde{\lambda^\R_i}(F^{i}_{j})$ the submatrix obtained by removing the column indexed by $i_j$.
    Note that $\det (\widetilde{\lambda^\R_i}(F^{i}_{j})) = \pm 1$.
    Then for a facet $\sigma = \cup_{i=1} ^{p} F^i_{j_i}$ of $K$, the determinant of the submatrix consisting of the columns of $\widetilde{\lambda^\R}$ corresponding to $\sigma$ is $\det(\widetilde{\lambda^\R_i}(F^{1}_{j_1})) \times \dots \times \det(\widetilde{\lambda^\R_i}(F^{p}_{j_p})|) = \pm 1.$
    Hence $\widetilde{\lambda^\R}$ is the $\{0,1\}$-lift of $\lambda^\R$, and it shows that every mod~$2$ characteristic map over the join of the boundaries of simplices has the $\{0,1\}$-lift.
\end{example}

\begin{lemma} \label{lemma:Choi-Park2016}\cite{Choi-Park2016}
    Let $A$ be an $n \times n$ matrix over $\Z$ whose determinant is odd.
    Then there is an $n \times n$ matrix $B$ over $\Z$ such that $\det(B)=\pm 1$ and $A \equiv B$ mod~$2$.
\end{lemma}

\begin{proposition} \label{proposition:lifting_is_D-J_class_property}
   The existence of a lift is a property of the D-J class.
\end{proposition}

\begin{proof}
Let $\lambda^\R$ and $\mu^\R$ be two D-J equivalent mod~$2$ characteristic maps over $K$.
There is an invertible matrix $A$ over $\Z_2$ such that $A \lambda^\R = \mu^\R$.
Suppose that $\widetilde{\lambda^\R}$ is a lift of $\lambda^\R$.
There is an integer matrix $\widetilde{A}$ such that $\widetilde{A} \equiv A$ mod~$2$ and the determinant of $\widetilde{A}$ is odd.
Lemma~\ref{lemma:Choi-Park2016} ensures that there is an invertible integer matrix $B$ such that
$B \widetilde{\lambda^\R} \equiv \widetilde{A} \widetilde{\lambda^\R} \equiv \mu^\R$ mod~$2$, that is, $B \widetilde{\lambda^\R}$ is a lift of $\mu^\R$ as well.
\end{proof}

For the sake of convenience, we define the \emph{dual complex} $\overline{K}$ of $K$ as the simplicial complex whose facets are the cofacets of $K$.
Also, we regard $\overline{\lambda}$ as a map from $[m]$ to $\Z^{m-n}$ such that $\overline{\lambda}(i)$ is the $i$th row of $\overline{\lambda}$, as we did for characteristic maps.
Then by the linear Gale duality, $\overline{\lambda}$ is a characteristic map over $\overline{K}$.

\begin{lemma} \label{lemma: dual lifting}
    Let $K$ be a simplicial complex.
    A mod~$2$ characteristic map $\lambda^\R$ over $K$ has a lift if and only if $\overline{\lambda^\R}$ has a lift as a mod~$2$ characteristic map over $\overline{K}$.
\end{lemma}

\begin{proof}
    Suppose that $\lambda^\R$ has a lift $\widetilde{\lambda^\R}$.
    Then $\overline{\widetilde{\lambda^\R}}$ is a characteristic map over $\overline{K}$.
    From the mod~$2$ reduction of the equation $\widetilde{\lambda^\R} \times \overline{\widetilde{\lambda^\R}} = \mathds{O}$, the mod~$2$ reduction of the columns of $\overline{\widetilde{\lambda^\R}}$ is a basis of $\ker \lambda^\R$.
    Hence up to D-J equivalence, $\overline{\lambda^\R} \equiv \overline{\widetilde{\lambda^\R}}$ mod~$2$, that is, $\overline{\widetilde{\lambda^\R}}$ is a lift of $\overline{\lambda^\R}$.

    The other direction is essentially the same.
\end{proof}

\begin{problem}[(toric) Lifting problem] \label{equiv_problem}
    Let $K$ be a PL~sphere.
    Does any mod~$2$ characteristic map over $K$ have a lift?
    Equivalently, does any mod~$2$ dual characteristic map over $K$ have a lift as a mod~$2$ characteristic map over $\overline{K}$.
\end{problem}

\subsection{Wedge operations}
Let $K$ be a simplicial complex on the vertex set $V$ and $\sigma$ a simplex in $K$.
The \emph{link} of $\sigma$ in $K$ is the simplicial complex defined by
$$
    \lk_K(\sigma) \coloneqq \{\tau \in K \mid \sigma \cup \tau \in K, \sigma \cap \tau = \varnothing\},
$$
and the \emph{deletion} of $\sigma$ in $K$ is the simplicial complex defined by
$$
    K \setminus \sigma \coloneqq \{\tau \in K \mid \sigma \not\subset \tau\}.
$$
For a singleton face $\{v\}$ of $K$, its link and deletion are denoted simply by $\lk_K (v)$ and $K \setminus v$, respectively.

For another simplicial complex $L$ on a disjoint vertex set from $K$, the \emph{join} $K \ast L$ of $K$ and $L$ is defined as the simplicial complex
$$
    K \ast L \coloneqq \{\sigma \cup \tau \mid \sigma \in K, \tau \in L \}.
$$
The \emph{suspension} of $K$ is given by
$$
    \susp(K) \partial I \ast K,
$$
where $I$ is a $1$-simplex with two new vertices $v_1$ and $v_2$, and $\partial I$ is its boundary complex.
In $\susp(K)$, the pair $\{v_1, v_2\}$ is referred to as a \emph{suspended pair}, and each vertex in it is called a \emph{suspended vertex}.

The \emph{wedge} of $K$ at a vertex $v$ of $K$ is defined as
$$
    \wed_v(K) \coloneqq (I \ast \lk_K(v)) \cup (\partial I \ast (K \setminus v)),
$$
where $I$ is a $1$-simplex comprising two new vertices.
It is evident that the link of a new vertex added after applying a wedge to $K$ is isomorphic to $K$.
In that sense, we often use $v_1$ and $v_2$ to refer to the two copies of $v$ in $\wed_v(K)$.
Consequently, $\wed_v(K)$ has vertex set $(V \setminus \{v\}) \cup \{v_1, v_2\}$.
Here, two vertices $v_1$ and $v_2$ are referred to as \emph{wedged vertices} of $v$, and the edge connecting them as the \emph{wedged edge} of $v$.
Notably, $\susp(K)$ can be viewed as a wedge at a ghost vertex of $K$.

The wedge operation can be defined equivalently as an easy combinatorial operation on the minimal non-faces of $K$: we duplicate the vertex $v$ in each minimal non-face of $K$ it appears in.
More precisely, let $\eta \subset V$ be a subset of the vertex set of $K$.
\begin{enumerate}
  \item If $\eta$ contains $v$, then $\eta$ is a minimal non-face of $K$ if and only if $\eta \setminus \{v\} \cup \{v_1, v_2 \}$ is a minimal non-face of $\wed_v(K)$.
  \item If $\eta$ does not contain $v$, then $\eta$ is a minimal non-face of $K$ if and only if $\eta$ is a minimal non-face of $\wed_v(K)$.
\end{enumerate}

As for suspensions, one can easily prove that the minimal non-faces of $K_1 \ast K_2$ is the union of the minimal non-faces of $K_1$ and $K_2$.
Then the minimal non-faces of $\partial I \ast K$ is obtained by adding $I$ in the minimal non-faces of $K$.
we can add a ghost vertex to $K$ which becomes a minimal non-face of $K$.
With this perspective, two consecutive wedge operations and join operations, including suspension, are associative and commutative with appropriate vertex identification.

Conversely, suppose that there are two vertices $v_1$ and $v_2$ such that for any minimal non-face $\eta$ of $K$, $\{v_1, v_2\} \subset \eta$ or $\{v_1, v_2\} \cap \eta = \varnothing$.
If $\sigma$ is a facet of $K$ containing neither $\{v_1\}$ nor $\{v_2\}$, then $\{v_1\} \cup \sigma$ is a non-face, so there is a minimal non-face $\eta$ of $K$ containing $v_1$.
This contradicts to the assumption.
Hence every facet of $K$ contains $v_1$ or $v_2$.
By the following lemma, if $\{v_1, v_2\}$ is not a minimal non-face of $K$, then it is a wedged edge of $K$, and otherwise, it is a suspended pair of $K$.

\begin{lemma} \cite{CP_wedge_2} \label{lemma:DCM}
    Let $K$ be a PL~sphere, and $v_1$ and $v_2$ be two vertices of $K$.
    If every facet of $K$ contains $v_1$ or $v_2$, then $K$ equals to either $\susp(L)$ with a suspended pair $\{v_1, v_2\}$, or $\wed_v(L)$ with wedged edge $\{v_1, v_2\}$ for some lower dimensional PL~sphere $L$.
\end{lemma}

\begin{corollary} \label{cor:repetition of rows}
    If a mod~$2$ dual characteristic map $\overline{\lambda^\R}$ over a PL~sphere $K$ satisfies $\overline{\lambda^\R}(v_1) = \overline{\lambda^\R}(v_2)$ for some vertices $v_1$ and $v_2$ of $K$, then $\{v_1, v_2\}$ is a suspended pair or a wedged edge of $K$.
\end{corollary}

\begin{proof}
  By the non-singularity of $\overline{\lambda^\R}$, every facet of $K$ contains $v_1$ or $v_2$.
  Then apply the previous lemma.
\end{proof}

Consider the vertex set of $K$ to be $[m]=\{1, \ldots, m\}$.
In light of the associative and commutative nature of wedge operations, we introduce the notation $K(J)$, termed a \emph{$J$-construction} of $K$ in \cite{BBCG2015}, for a positive integer $m$-tuple $J=(j_1, j_2, \ldots, j_m)$.
This represents the simplicial complex obtained by applying multiple wedge operations to $K$; for each $i \in [m]$, wedge operations are applied $j_i-1$ times to $K$ at $i$ or its copied vertices.
We will often denote the copied vertices of $i$ by $i_1, i_2, \ldots, i_{j_i}$.
For the sake of convenience, even when $j_i = 1$, we treat $i$ as $i_1$, and we say $i_k$ is a wedged vertex of $i$ for each $k \geq 1$.

In addition, due to the commutativity and associativity of the operations involved, we have the relationship:
\begin{equation} \label{eq:J-construction of suspension}
  (\partial I \ast K)(j_1, j_2, \ldots, j_{m+2}) = \partial I(j_1, j_2) \ast K(j_3, \ldots, j_{m+2}),
\end{equation}
where $I$ is the $1$-simplex on $\{1, 2\}$.
This leads to two characterizations regarding the suspension and wedge operations.
\begin{proposition} \label{proposition:suspension wedge}
    Let $K$ be a simplicial complex, $I$ a $1$-simplex, $v$ a vertex of $K$. Then:
    \begin{enumerate}
      \item $K$ is a suspension if and only if so is $\wed_v(K)$ for any non-suspended vertex $v$ of $K$,
      \item $K$ is a wedge if and only if so is $\Sigma(K)$.
    \end{enumerate}
\end{proposition}

We define the \emph{Picard number} of $K$ as $m-n$.
One can observe that the wedge operation preserves the Picard number of $K$ whereas the suspension increases the Picard number of $K$ by~$1$.
It is known that the link, wedge, and suspension operations are closed within the class of PL~spheres, see \cite{Choi-Jang-Vallee2023} for details.

A PL~sphere not isomorphic to some wedge of another PL~sphere is termed a \emph{seed}.
It should be noted that any PL~sphere $K$ of Picard number~$p$ can be written as $L(J)$, where $L$ is a seed of the same Picard number~$p$.
In addition, one can easily see that $L$ is uniquely determined up to isomorphism, whereas $J$ can be different.

For our purpose, we are interested in PL~spheres of dimension $n-1$ on $[m]$ whose real Buchstaber number coincides with their Picard number $m-n$.
Such a PL~sphere is said to be \emph{($\Z_2^n$-)colorable}.
Ewald \cite{ewald1996combinatorial} observed that all colorable PL~spheres are obtained by colorable seeds, and Choi and Park \cite{CP_wedge_2} proved that the number of colorable seeds with given Picard number is finite.
Although obtaining the list of colorable seeds of given Picard number is a difficult problem in itself, the list up to Picard number~$4$ has been established in \cite{Choi-Jang-Vallee2023}.

\begin{theorem} \cite{Choi-Jang-Vallee2023} \label{theorem: classification}
    The number of colorable seeds with Picard number at most $4$ up to isomorphism is as follows:
    \begin{center}
        \begin{tabular}{l*{12}{c}r}
        \toprule
        $p\backslash n$ &$1$ & $2$ & $3$ & $4$ & $5$ & $6$ & $7$ & $8$ & $9$ & $10$ & $11$ & $>11$& total \tabularnewline \midrule
        $1$&$1$&  &&&&&&&&&&&$1$\tabularnewline
        $2$ && $1$ & & &&&&&&&&& $1$\tabularnewline
        $3$ && $1$ & $1$& $1$& & & & &&&& &$3$\tabularnewline
        $4$ && $1$ & $4$ & $21$ & $142$ & $733$ & $1190$ & $776$ & $243$& $39$ & $4$ & &$3153$\tabularnewline
        \bottomrule
        \end{tabular}
    \end{center}
    with the empty slots displaying zero.
\end{theorem}

\section{Toric wedge induction}
\subsection{Toric wedge induction and its modification}

Let $K$ be an $(n-1)$-dimensional PL~sphere on $[m]$.
There are operations on mod~$2$ characteristic maps over $K$ corresponding to wedge and link operations on $K$.
Let $\Lambda^\R$ be a mod~$2$ characteristic map over $\wed_{v}(K)$.
Up to D-J equivalence and vertex relabeling, we may assume that
\begin{equation} \label{eq:wedged char}
 \Lambda^\R =
    \kbordermatrix{ & v_1 & v_2 & &  \\
                    & 1 & 0 & \multirow{2}{*}{$\mathds{O}$} & \textbf{a} \\
                    & 0 & 1 & & \textbf{b} \\
                    & \multicolumn{2}{c}{$\mathds{O}$} & I_{n-1} & A},
\end{equation}
where the column indexes $v_1$ and $v_2$ stand for the associated wedged vertices, $\textbf{a}$ and $\textbf{b}$ are row vectors of size $m-n$, $I_{n-1}$ is the identity matrix of size $n-1$, and $A$ is a $\Z_2$-matrix of size $(n-1)\times (m-n)$.
The \emph{projection} of $\Lambda^\R$ with respect to a face $\sigma$ of $\wed_v(K)$ is a map from the vertex set of $\lk_K(\sigma)$ to $\Z_2^{n-|\sigma|}$ defined by
\begin{equation} \label{eq: projection}
    \proj_\sigma (\Lambda^\R)(w) = [\Lambda^\R(w)] \in \faktor{\Z_2^n}{\left< \Lambda^\R(v) \mid v \in \sigma \right>} \cong \Z_2^{n-|\sigma|}
\end{equation}
for each vertex $w$ of $\lk_K(\sigma)$.
If we fix a basis of $\Z_2^{n-|\sigma|}$, we can see that $\proj_\sigma (\Lambda^\R)$ is a mod~$2$ characteristic map over $\lk_K(\sigma)$.
We call $\proj_\sigma (\Lambda^\R)$ the \emph{projection} onto $\lk_K(\sigma)$.

The links $\lk_{\wed_{v}(K)}(v_1)$ and $\lk_{\wed_{v}(K)}(v_2)$ are isomorphic to $K$ by identifying $v_2$ and $v_1$ with $v$, respectively, and the projections of $\Lambda^\R$ with respect to $v_1$ and $v_2$ are written as
$$
    \lambda^\R_1 = \proj_{v_1}(\Lambda^\R) =\kbordermatrix{ & v_2 & &  \\
                                         & 1 & \mathds{O} & \textbf{b} \\
                                         & \mathds{O} & I_{n-1} & A}, \text{ and }
$$
$$
    \lambda^\R_2 = \proj_{v_2}(\Lambda^\R) =\kbordermatrix{ & v_1 & &  \\
                                        & 1 & \mathds{O} & \textbf{a} \\
                                        & \mathds{O} & I_{n-1} & A}.
$$
If we consider the first column of each matrix corresponds to the vertex $v$, then two matrices $\lambda^\R_1$ and $\lambda^\R_2$ are mod~$2$ characteristic maps over $K$.
Hence, $\Lambda^\R$ corresponds to a choice of two mod~$2$ characteristic maps over $K$ whose first $n$ columns form an identity matrix such that their submatrices formed by deleting the first row and the first column are identical up to D-J equivalence.

Conversely, one may construct at most one $\Lambda^\R$ over $\wed_{v}(K)$ from an ordered pair of mod~$2$ characteristic maps $\lambda_1^\R$ and $\lambda_2^\R$ over $K$ as in~\eqref{eq:wedged char}.
We denote $\Lambda^\R = \lambda_1^\R \wedge_v \lambda_2^\R$ if it exists.
If $\lambda_1^\R = \lambda_2^\R = \lambda^R$, then $\lambda^\R \wedge_v \lambda^\R$ always exists for any $v$, and it is called the \emph{canonical extension} of $\lambda^\R$ at $v$.

It should be noted that we can represent each characteristic map over $K(J)$ by a combination of characteristic maps over $K$.
From this viewpoint, we shall introduce one powerful inductive tool to demonstrate some properties on real toric spaces, for example the existence of a lift of $\Lambda^\R$ as in this paper.

For a PL~sphere $K$, let $\cX$ be a collection of pairs $(L, \lambda^\R)$ such that every $L$ is expressed as $K(J)$ for some $J$, and $\lambda^\R$ is a mod~$2$ characteristic map over $L$.
Then $\cX$ is called a \emph{wedge-stable set based on $K$} if  $(\wed_v(L), \lambda_1^\R \wedge \lambda_2^\R) \in \cX$ whenever both $(L, \lambda_1^\R)$ and $(L, \lambda_2^\R)$ are in $\cX$.

We present the concept of \emph{toric wedge induction} which is a method employed to demonstrate the validity of a given property across $\cX$.

\begin{proposition}[Toric wedge induction] \label{TWI}
    For a PL~sphere $K$, let $\cX$ be a wedge-stable set based on $K$, and $\cP$ a property.
    Suppose that the following holds;
    \begin{enumerate}
      \item \textbf{Basis step:} All $(K, \lambda^\R) \in \cX$ satisfies $\cP$ .
      \item \textbf{Inductive step:} If $(L, \lambda_1^\R)$, $(L, \lambda_2^\R) \in \cX$ satisfy $\cP$, then so does $(\wed_v(L), \lambda_1^\R \wedge_v \lambda_2^\R)$ for any vertex $v$ of $L$.
    \end{enumerate}
    Then $\cP$ holds on $\cX$.
\end{proposition}

The credit of original idea of toric wedge induction should be given to Choi and Park~\cite{Choi-Park2016}.
They used it for showing the projectivity of certain toric manifolds in  ~\cite{Choi-Park2016} or \cite{CP_CP_variety}. 
Later, the authors of this paper used it for classifying toric manifolds satisfying equality within an inequality regarding the number of minimal components in their rational curve space \cite{Choi-Jang-Vallee2023}.

However, it is sometimes challenging to perform the inductive step.
In that situation, we can relax it by strengthening the basis step.
In this paper, we introduce a new, easier version of this method that helps with the toric lifting property we want to show.

In order to do it, we briefly review the notions and properties, following \cite{CP_wedge_2}, where the reader may find a much more details about the relations between characteristic maps over $K(J)$ and puzzles explained below.

The \emph{pre-diagram $D'(K)$} of $K$ is an edge-colored non-simple graph such that
\begin{enumerate}
    \item the node set of $D'(K)$ is the set of D-J classes of the mod~$2$ characteristic maps over $K$,
    \item for a vertex $v$ of $K$ and two mod~$2$ characteristic maps $\lambda^\R_1$ and $\lambda^\R_2$ over $K$, a pair $(\{\lambda^\R_1, \lambda^\R_2\}, v)$ is a colored edge of $D'(K)$ if and only if there is a mod~$2$ characteristic map over $\wed_{v}(K)$ whose two projections onto $K$ are $\lambda^\R_1$ and $\lambda^\R_2$.
\end{enumerate}

We denote $G(J)$ the $1$-skeleton of the simple polytope $P(J) \coloneqq \Delta^{j_1-1} \times \Delta^{j_2-1} \times \dots \times \Delta^{j_m-1}$.
Each edge $\bm{\epsilon}$ of $G(J)$ is uniquely written as 
\begin{align*}
    \bm{\epsilon} = \alpha_1 \times \alpha_2 \times \dots \times \alpha_{v-1} \times \epsilon_v \times \alpha_{v+1} \times \dots \times \alpha_m,
\end{align*}
where $\alpha_i$ is a vertex of $\Delta^{j_i-1}$, $1 \leq i \leq m, i \not= v$, and $\epsilon_v$ is an edge of $\Delta^{j_v-1}$.
Then color $\bm{\epsilon}$ by $v$.

Then, a mod~$2$ characteristic map $\lambda^\R$ over $K(J)$ can be expressed by an edge-colored graph homomorphism $\phi \colon G(J) \to D'(K)$;
When $\bm{\alpha}$ is a vertex of $G(J)$, we can write
\begin{align*}
  \bm{\alpha} = \alpha_1 \times \alpha_2 \times \dots \times \alpha_m,
\end{align*}
where $1 \leq \alpha_i \leq j_i$ is a vertex of $\Delta^{j_i-1}$ for $1 \leq i \leq m$.
Observe that
\begin{align*}
  \mathcal{F}(\bm{\alpha}) \coloneqq \{1_1, 1_2, \ldots, 1_{j_1}, \ldots, m_1, m_2, \ldots, m_{j_m}\} \setminus \{1_{\alpha_1}, 2_{\alpha_2}, \ldots, m_{\alpha_m} \}
\end{align*} does not contain any minimal non-face of $K(J)$, so it is a face of $K(J)$.
Since each vertex $i_{j_k}$ ($j_k \not= \alpha_i$) in $\mathcal{F}(\bm{\alpha})$ is a wedged vertex of $i_{\alpha_i}$, $\lk_{K(J)}(\mathcal{F}(\bm{\alpha}))$ is isomorphic to $K$, and the projection $\proj_{\mathcal{F}(\bm{\alpha})}(\lambda^R)$ is a mod~$2$ characteristic map over $K$ by the natural bijection between $\{1, 2, \ldots, m\}$ and $\{1_{\alpha_1}, 2_{\alpha_2}, \ldots, m_{\alpha_m}\}$.
Define $\phi$ by $\phi(\bm{\alpha})=\proj_{\mathcal{F}(\bm{\alpha})}(\lambda^\R)$ for each vertex $\bm{\alpha}$ of $G(J)$.
Let $\epsilon$ be an edge of $G(J)$.
Up to relabeling the vertices of $K$, $\epsilon$ consists of two vertices $\alpha$ and $\alpha' \coloneqq \alpha'_1 \times \alpha_2 \times \dots \times \alpha_m$.
Then $\phi(\epsilon) = \{\proj_{\mathcal{F}(\bm{\alpha})}(\lambda^\R), \proj_{\mathcal{F}(\bm{\alpha'})}(\lambda^\R) \}$.
By the following proposition, $\overline{\proj_{\mathcal{F}(\bm{\alpha})}(\lambda^\R)}$ and $\overline{\proj_{\mathcal{F}(\bm{\alpha'})}(\lambda^\R)}$ are same except their first rows.
Hence $\phi$ is an edge-colored graph homomorphism.

\begin{remark} \label{rmk:reducible}
    In the above situation, it is worthy to note that if $\overline{\lambda^\R}(1_{\alpha_1}) = \overline{\lambda^\R}(1_{\alpha'_1})$, then  $\phi$ is not irreducible.
\end{remark}

\begin{proposition} \label{proposition: dual projection}
    Let $K$ be a PL~sphere, and $\sigma$ a face of $K$ such that the Picard numbers of $K$ and $\lk_K(\sigma)$ are same.
    Up to D-J equivalence, the dual of the projection of a mod~$2$ characteristic map $\lambda^\R$ with respect to a face $\sigma$ of $K$ is obtained by removing rows corresponding to $\sigma$ in $\overline{\lambda^\R}$.
\end{proposition}

\begin{proof}
    It is sufficient to prove the result for a vertex $\sigma = \{v\}$.
    To project $\lambda^\R$ with respect to $\{v\}$, let us left multiply $\lambda^\R$ by an invertible matrix $g$ so that the vector $g\lambda^\R(v)$ has a single nonzero entry, say at index $k$.
    From \eqref{eq: projection}, if we delete the $v$th column and $k$th row of $g\lambda^\R$, then we obtain the projection $\proj_{v}(\lambda^\R)$.
    Let $\overline{\mu^\R}$ be the matrix obtained by removing $v$th row in $\overline{\lambda^\R}$.
    Then, any column of $\overline{\mu^\R}$ is an element of the kernel of $\proj_{v}(\lambda^\R)$ since every component of the $v$th column of $\lambda^\R$ except the $k$th one is $0$.
    Note that $\overline{\mu^\R}$ remains of full rank since there must exist a cofacet of $K$ that does not contain $v$.
\end{proof}

However, not all edge-colored graph homomorphism $\phi$ is obtained from a mod~$2$ characteristic map over $K(J)$.
If it is, $\phi$ is called a \emph{(realizable) puzzle} over $K(J)$ (or, over $K$, if there is no confusion), and $\lambda_\phi^\R$ denotes the corresponding mod~$2$ characteristic map over $K(J)$.
In particular, a puzzle which does not contain any edge corresponding to a canonical extension is called \emph{irreducible}.

Consider a realizable puzzle $\phi$ over $K(J)$.
In $G(J)$, two edges \begin{align*} \bm{\epsilon} &= \alpha_1 \times \alpha_2 \times \dots \times \alpha_{v-1} \times \epsilon_v \times \alpha_{v+1} \times \dots \times \alpha_m, \\
                                      \bm{\epsilon}' &= \alpha'_1 \times \alpha'_2 \times \dots \times \alpha'_{v-1} \times \epsilon'_v \times \alpha'_{v+1} \times \dots \times \alpha'_m  \end{align*}
are called \emph{parallel} if $\epsilon_v = \epsilon'_v$.
By \cite[Corollary~4.4]{CP_wedge_2}, if there is an edge $\bm{\epsilon}$ of $\phi$  corresponding to a canonical extension, then so does every parallel edge to $\bm{\epsilon}$.
Therefore, every puzzle is obtainable from an irreducible puzzle by a sequence of canonical extensions.

\begin{proposition}[Modified toric wedge induction] \label{MTWI}
    For a PL~sphere $K$, let $\cX$ be a wedge-stable set based on $K$, and $\cP$ a property.
    Suppose that the following holds;
    \begin{enumerate}
      \item \textbf{Basis step:} For any positive integer tuple $J$ and any irreducible realizable puzzle $\phi$ over $K(J)$, $(K(J),\lambda_\phi^\R)$ satisfies $\cP$.
      \item \textbf{Inductive step:} If $(L, \lambda^\R) \in \cX$ satisfies $\cP$, then so does the pair consisting of the wedge of $L$ at $v$ and the canonical extension of $\lambda^\R$ at $v$ for any vertex $v$ of $L$.
    \end{enumerate}
    Then $\cP$ holds on $\cX$.
\end{proposition}

It should be noted that the basis step consists of finitely many cases as the lemma below.

\begin{lemma} \label{lemma:finite irreducible}
    For a simplicial complex $K$, there are finitely many irreducible puzzles over $K$.
\end{lemma}

\begin{proof}
    Let $J=(j_1, j_2, \ldots, j_m)$.
    Fix a vertex $\bm{\alpha}=\alpha_1 \times \alpha_2 \times \dots \times \alpha_m$ of $P(J)$.
    For each $j_v$, $\bm{\alpha}$ is also a vertex of the simplex $\alpha_1 \times \alpha_2 \times \dots \times \alpha_{v-1} \times \Delta^{j_v-1} \times \alpha_{v+1} \times \dots \times \alpha_m$.
    Because any two vertices of $\Delta^{j_v-1}$ forms an edge, the number $j_v$ can not exceed the number of mod~$2$ characteristic maps over $K$.
\end{proof}

\begin{remark}
    The concept of wedge of characteristic maps and puzzle is not only described for mod~$2$ characteristic maps, but also for characteristic maps. See \cite{CP_wedge_2} for details.
    This means that we can apply toric wedge induction to a collection of toric spaces.
    In addition, the number of PL~spheres which admit a characteristic map is less than or equal to the number of PL~spheres which admit a mod~$2$ characteristic map, so it is finite.
    However, the number of characteristic maps over a seed is not finite.
    Hence the basis step may not be implemented by direct computations in finite steps.
\end{remark}

\subsection{Modified Toric wedge induction in terms of dual characteristic maps}
Even if the modified version (Proposition~\ref{MTWI}) of the toric wedge induction has an easier inductive step than the original version (Proposition~\ref{TWI}), there still remains a challenging part to deal with: constructing (irreducible) puzzles.
In this subsection, we characterize irreducible puzzles in terms of dual characteristic maps over seeds, and restate the modified toric wedge induction based on a seed using dual characteristic maps instead of irreducible puzzles.

Let $K$ be a colorable seed of dimension $n-1$ on $[m]$, $\lambda^\R$ a mod~$2$ characteristic map over $K$.
Assume that there are two vertices of $K$ such that $\overline{\lambda^\R}(v) = \overline{\lambda^\R}(w)$.
By Remark~\ref{cor:repetition of rows}, $K$ has to be a suspension, and then by Proposition~\ref{proposition:suspension wedge}, $K$ is the suspension of a seed.
Let us call a seed that is a suspension a \emph{suspended seed}.
Hence if $K$ is a non-suspended seed, then every dual characteristic map over $K$ is injective.
An \emph{IDCM} denotes an injective mod~$2$ dual characteristic map.
By the following lemma, every irreducible puzzle over a non-suspended seed corresponds to an IDCM.

\begin{lemma} \label{lemma:irreducible IDCM}
    Let $\phi$ be an irreducible puzzle over $K(J)$ for a PL~sphere $K$ on $[m]$ and a positive integer $m$-tuple $J$.
    Then $\overline{\lambda_\phi^\R}$ is injective if and only if $\overline{\phi(\bm{\alpha})}$ is injective for any vertex $\bm{\alpha}$ of $G(J)$.
\end{lemma}

\begin{proof}
    Assume that $\overline{\lambda_\phi^\R}$ is injective.
    By Proposition~\ref{proposition: dual projection}, each projection $\overline{\phi(\bm{\alpha})}$ of $\lambda_\phi^\R$ has no repeated rows, so it is injective.

    Conversely, suppose that $\overline{\lambda_\phi^\R}(v) = \overline{\lambda_\phi^\R}(w)$ for some vertices $v$, $w$ of $K(J)$.
    By Remark~\ref{rmk:reducible} and the irreducibility of $\phi$, $\{v, w\}$ cannot be copies of one vertex of $K$ by wedge operations.
    Hence there is a vertex $\bm{\alpha}$ of $G(J)$ such that $v = k_{\alpha_k}$ and $w = l_{\alpha_{l}}$ for some $\alpha_k \leq j_k$, $\alpha_l \leq j_l$, and distinct vertices $k, l \in [m]$ and .
    This yields $\overline{\proj_{\cF(\bm{\alpha})}(\lambda_\phi^\R)}(v) = \overline{\proj_{\cF(\bm{\alpha})}(\lambda_\phi^\R)}(w)$, so $\overline{\phi(\alpha)}$ is not injective.
\end{proof}

In general, a seed $K$ is of the form $\partial I_1 \ast \dots \ast \partial I_q \ast L$, where $I_k$ is the $1$-simplex with vertices $2k-1$ and $2k$ for each $k \leq q$, and $L$ is a non-suspended seed.
By \eqref{eq:J-construction of suspension}, \begin{equation}\label{eq:J-construction}
                                                     K(J) = \partial I_1(J_1) \ast \dots \ast \partial I_q(J_q) \ast L(J_{q+1}).
                                                   \end{equation}
                                                                                                   
Before studying mod~$2$ characteristic maps over $K(J)$, we need the following analysis of mod~$2$ characteristic maps over the join of two simplicial complexes.
Research on (integral) characteristic maps over the join $K*L$ is well-established in references such as \cite{Choi-Park2016IJM} and \cite{Dobrinskaya2001},  and it can be converted well to mod~$2$ characteristic maps. 
Refer to these for further information.
Consider the join of simplicial complexes $K_1$ of dimension $n_1-1$ with $m_1$ vertices and $K_2$ of dimension $n_2-1$ with $m_2$ vertices.
One can observe that any (mod~$2$) characteristic map $\lambda^\R$ over $K_1 \ast K_2$ has the following form;
$$
  \lambda^\R = \begin{bmatrix} \lambda^\R_{11}  &  \lambda^\R_{12} \\
                       \lambda^\R_{21} &  \lambda^\R_{22}                                                                \end{bmatrix},
$$
where $\lambda^\R_{11}$, $\lambda^\R_{22}$ are mod~$2$ characteristic maps over $K_1$, $K_2$ respectively, see \cite[Lemma~3.1]{Choi-Park2016IJM} for instance.
Moreover, we can assume that the first $n_1$ columns of $\lambda^\R_{11}$, and the first $n_2$ columns of $\lambda^\R_{22}$ form identity matrices by D-J equivalence and vertex relabeling.
Then, up to D-J equivalence,
\begin{equation} \label{eq:joining representative}
    \lambda^\R = \begin{bmatrix}  I_{n_1} & A  & \mathds{O} & B \\
                                    \mathds{O} & C & I_{n_2} & D \end{bmatrix}
    \text{ and }
    \overline{\lambda^\R} = \begin{bmatrix}
                                     A & B \\
                                     I_{m_1-n_1} & \mathds{O} \\
                                     C & D \\
                                     \mathds{O} & I_{m_2-n_2} 
                                   \end{bmatrix}.
\end{equation}
We call this form of $\lambda^\R$ the \emph{joining representative} of $\lambda^\R$.
It should be noted that $\begin{bmatrix}
                           A \\
                           I_{m_1 - n_1}  
                         \end{bmatrix}$ and $\begin{bmatrix}
                                               D \\
                                               I_{m_2 - n_1} 
                                             \end{bmatrix}$ are mod~$2$ dual characteristic maps over $K_1$ and $K_2$, respectively.

\begin{lemma} \label{lemma:non-injectivity}
    Let $K$ be a seed and $\phi$ an irreducible puzzle over its $J$-construction~\eqref{eq:J-construction}.
    Suppose that there are two vertices $v$ and $w$ of $K(J)$ such that $\overline{\lambda_\phi^\R}(v) = \overline{\lambda_\phi^\R}(w)$.
    Then there exists a suspended pair $\{2k-1, 2k\}$ of $K$ such that $v = (2k-1)_{s}$, $w = (2k)_{t}$ for some $s \leq j_{2k-1}$ and $t \leq j_{2k}$.
\end{lemma}
    
\begin{proof}
    By Remark~\ref{cor:repetition of rows}, $\{v, w\}$ is a suspended pair or a wedged edge of $K(J)$.
    In the former case, we are done.
    
    Let $\{v, w\}$ be a wedged edge of $K(J)$.
    As written in the proof of Lemma~\ref{lemma:irreducible IDCM}, there are two distinct vertices $l$ and $l'$ of $K$ such that $v = l_{s}$, $w = l'_{t}$ for some $s \leq j_l$ and $t \leq j_{l'}$.
    By commutativity of wedge and join operations, $\{v, w\}$ is a wedged edge of $L(J_{q+1})$ or $\partial I_k (J_k)$ for some $k$.
    If $l$ and $l'$ are contained in some $\partial I_k (J_k)$, then they form a suspended pair since they are distinct vertices.
    
    Assume that $l$ and $l'$ are contained in $L$.
    Let $\bm{\alpha}$ be a vertex of $G(J)$ such that its $l$th and $l'$th components are $s$ and $t$, respectively.
    Then $\phi(\bm{\alpha})$ is a mod~$2$ characteristic map over $K$ such that $\overline{\phi(\bm{\alpha})}(l) = \overline{\phi(\bm{\alpha})}(l')$.
    Consider the joining representative \eqref{eq:joining representative} of $\phi(\bm{\alpha})$ over $K$ by setting $K_1 = \partial I_1 \ast \dots \ast \partial I_q$ and $K_2 = L$.
    Then the mod~$2$ dual characteristic map $\begin{bmatrix}
                 D \\
                 I_{m_2-n_2} 
               \end{bmatrix}$ over the non-suspended seed $L$ has repeated two rows at $l$ and $l'$, which contradicts that a non-suspended seed only admits IDCMs.
\end{proof}

Suppose that $\phi$ is an irreducible puzzle over $K(J)$, and there are two vertices $v$ and $w$ of $K(J)$ such that $\overline{\lambda^\R_\phi}(v) = \overline{\lambda^\R_\phi}(w)$.
By Lemma~\ref{lemma:non-injectivity}, there is a suspended pair of $\{2k-1, 2k\}$ of $K$ such that $v=(2k-1)_s$ and $w=(2k)_t$.
If $\{v, w\}$ is not a suspended pair, that is $J_k \not=(1,1)$, then it is a wedged edge of $K(J)$ by Corollary~\ref{cor:repetition of rows}.
Without loss of generality, we can assume that $j_{2k-1} \geq 2$.
The irreducibility of $\phi$ ensures $\overline{\lambda^\R_\phi}((2k-1)_1) \not= \overline{\lambda^\R_\phi}((2k-1)_2)$.

Consider a positive integer tuple $J' = (j'_1, \dots, j'_m)$ such that $j'_{2k-1} = 1$, $j'_{2k}  = j_{2k-1}+j_{2k}-1$, and $j'_i = j_i$ for $i \not=2k-1, 2k$.
Define a simplicial map $f \colon K(J) \to K(J')$ by \begin{equation}
                                                       f(x) = \begin{cases}
                                                             (2k)_{j_{2k}+l-1}, & \mbox{if } x = (2k-1)_l \mbox{ for } l>1 \\
                                                             x, & \mbox{otherwise}.
                                                           \end{cases}
                                                     \end{equation}
Since any J-constructions of $\partial I_k$ is a simplex, $f$ is an isomorphism.
Then $\lambda^\R_\phi \circ f^{-1}$ is a mod~$2$ characteristic map over $K(J')$.
Note that this does not correspond to an irreducible puzzle over $K(J')$.
Hence we obtain the following theorem.

\begin{theorem}
  Let $K$ be a seed, and $\phi$ an irreducible puzzle over $K(J)$.
  Suppose that there are two vertices $v$ and $w$ of $K(J)$ such that $\overline{\lambda_\phi^\R}(v) = \overline{\lambda_\phi^\R}(w)$.
  If $\{v, w\}$ is not a suspended pair, then there exists an isomorphism $f \colon K(J') \to K(J)$ for a positive integer tuple $J'$ such that $\lambda_\phi \circ f$ does not correspond to an irreducible puzzle over $K(J')$.
\end{theorem}

We call $\overline{\lambda^\R}$ \emph{quasi-injective} if $\overline{\lambda^\R}(v) =\overline{\lambda^\R}(w)$ implies that $\{v, w\}$ is a suspended pair of $K$. 
From the above theorem, we can restate the modified toric wedge induction based on a seed as the following.

\begin{proposition} \label{MTWI2}
    Let $K$ be a colorable seed, and
    $$
        \cX_K=\{(L, \lambda^\R) \mid L=K(J) \text{ for some } J, \ \lambda^\R \text{ is a mod~$2$ characteristic map over } L \}.
    $$
    For a property $\cP$, suppose that the following holds;
    \begin{enumerate}
      \item \textbf{Basis step:} For any positive integer tuple $J$ and any quasi-IDCM $\overline{\lambda^\R}$ over $K(J)$, $(K(J),\lambda^\R)$ satisfies $\cP$.
      \item \textbf{Inductive step:} If $(L, \lambda^\R) \in \cX_K$ satisfies $\cP$, then so does the pair consisting of the wedge of $L$ at $v$ and the canonical extension of $\lambda^\R$ at $v$ for any vertex $v$ of $L$.
    \end{enumerate}
    Then $\cP$ holds on $\cX_K$.
\end{proposition}

Even though Lemma~\ref{lemma:finite irreducible} ensures that the basis step of modified toric wedge induction is a finite problem, we can see directly that in this form.
The number of suspended pairs can not exceed the Picard number, so the number of rows of a quasi-IDCM can not exceed $p+2^p-1$.

\begin{remark}
In particular, (modified) toric wedge induction is useful when we want to see a property for all real toric spaces over PL~spheres of Picard number~$p$.
By the injectivity of mod~$2$ dual characteristic maps over non-suspended seeds, if $K$ is a non-suspended seed, then we have $m \leq 2^p-1$, so there are finitely many non-suspended seeds of Picard number~$p$.
Since suspended seeds of Picard number~$p$ are suspensions of non-suspended and suspended seeds of Picard number~$p-1$, there are finitely many suspended seeds of Picard number~$p$ as well.
Hence, Lemma~\ref{lemma:finite irreducible} guarantees that the basis steps of (modified) toric wedge induction based on all seeds of Picard number~$p$ can be also solved in finite time.
\end{remark}

\section{Proof of the main theorem}
For the reader's convenience, we recall the statement of the main theorem.
\main*

\begin{proof}

Let $r$ be the rank of a subgroup acting freely on $\R \cZ_K$.
For the case $r \leq 3$, the statement is verified by Theorem~\ref{thm: r3}.
Hence, it is enough to consider the case when $m=n+4$, and $r=m-n=4$.

We apply the modified version of toric wedge induction with $\cP$ as in Proposition~\ref{MTWI2}.
Let $K$ be a seed.
For a positive integer tuple $J$ and a mod~$2$ characteristic map $\lambda^\R$ over $K(J)$, we say the pair $(K(J), \lambda^\R)$ satisfies $\cP$ if and only if the mod~$2$ characteristic map $\lambda^\R$ has a lift.
The basis step of any seed of Picard number~$4$ will be accomplished in Section~\ref{section: basis step 1}.
The inductive step follows from Lemma~\ref{lemma:inductive step} below.
Therefore, by Proposition~\ref{MTWI2}, $\cP$ holds on the set of all real toric spaces of Picard number~$4$.
\end{proof}

\begin{lemma}[Inductive Step] \label{lemma:inductive step}
    For a PL~sphere $K$, if a mod~$2$ characteristic map $\lambda^\R$ over $K$ has a lift, then the canonical extension of $\lambda^\R$ at $v$ has a lift for any vertex $v$ of $K$.
\end{lemma}

\begin{proof}
    Let $K$ be an $(n-1)$-dimensional PL~sphere on $[m]$.
    By relabeling the vertices and D-J equivalence, we may assume that the first $n$ columns of $\lambda^\R$ and the last $m-n$ columns of $\overline{\lambda}^\R$ are of form the identity matrix, and the wedge operation is performed at $v=1 \in [m]$.
    Let $\lambda=\widetilde{\lambda^\R}$ be a lift of $\lambda^\R$.
    By Proposition~\ref{proposition:lifting_is_D-J_class_property}, up to D-J equivalence, we can also assume that its first $n$ columns form the $n \times n$ identity matrix.

    Set $\lambda^\R = \begin{bmatrix}
                     \multirow{4}{*}{$I_n$}  & \overline{\lambda^\R}(1) \\
                      & \overline{\lambda^\R}(2) \\
                      & \vdots \\
                      & \overline{\lambda^\R}(n)
                   \end{bmatrix}$
    and its lift $\lambda = \begin{bmatrix}
                     \multirow{4}{*}{$I_n$}  & \overline{\lambda}(1) \\
                      & \overline{\lambda}(2) \\
                      & \vdots \\
                      & \overline{\lambda}(n)
                   \end{bmatrix}$.
    Then, 
    $$
       \Lambda^\R = \begin{bmatrix} \multirow{5}{*}{$I_{n+1}$}  & \overline{\lambda^\R}(1) \\
                                                                        & \overline{\lambda^\R}(1) \\
                                                                        & \overline{\lambda^\R}(2) \\
                                                                        & \vdots \\
                                                                        & \overline{\lambda^\R}(n)
                                                                \end{bmatrix} \text{ and }
        \Lambda = \begin{bmatrix} \multirow{5}{*}{$I_{n+1}$}  & \overline{\lambda}(1) \\
                                                                        & \overline{\lambda}(1) \\
                                                                        & \overline{\lambda}(2) \\
                                                                        & \vdots \\
                                                                        & \overline{\lambda}(n)
                                                                \end{bmatrix}
    $$
    are the canonical extension of $\lambda^\R$ at $1$ and its lift, respectively.
\end{proof}

\section{The basis step} \label{section: basis step 1}
In this section, we prove that every quasi-IDCM over a PL~sphere $K$ of Picard number~$4$ has a lift as a mod~$2$ characteristic map over $\overline{K}$.
Then the basis step of Proposition~\ref{MTWI2} is accomplished by Lemma~\ref{lemma: dual lifting}.

Let $K$ be an $n-1$ dimensional PL~sphere of Picard number~$4$.
Suppose that $K$ admits a quasi-IDCM $\lambda^\R$ that is not an IDCM.
Then if $K = L(J)$ for some seed $L$, then by Theorem~\ref{theorem: classification}, $L$ is one of the following: \begin{enumerate}
                                                                          \item $L_1 = \partial I \ast \partial P_5$,
                                                                          \item $L_2 = \partial I_1 \ast \partial I_2 \ast \partial I_3 \ast \partial I_4$,
                                                                          \item $L_3 = \partial I \ast \partial C^4(7)$,
                                                                        \end{enumerate}
where $P_5$ is a pentagon, and $C^4(7)$ is a $4$-dimensional cyclic polytope with $7$ vertices.
By Example~\ref{example: general Bott}, if $K = L_2(J)$, then every mod~$2$ characteristic map over $K$ has a lift, so assume that $L \not= L_2$ in addition.
Since there is no seed of Picard number~$2$ with $n=1$ and $3$ by Theorem~\ref{theorem: classification}, we can see that $\partial P_5$ and $\partial C^4(7)$ are non-suspended seeds.
Hence $\lambda^\R$ has exactly one pair of vertices with row repetition.

For convenience, let $L' = \partial P_5$ or $ \partial C^4(7)$. 
By \eqref{eq:J-construction of suspension}, we can write $K = L(J) = \partial I \ast L'(J_2)$.
With setting $K_1 = \partial I$ and $K_2 = L'(J_2)$, $\lambda^\R$ is of the form by \eqref{eq:joining representative}: \begin{equation*}
    \lambda^\R = \begin{bmatrix}  1 & 1  & \mathds{O} \\
                                    \mathds{O} & \ast & \mu^\R \end{bmatrix},
      \end{equation*}
where $\mu^\R$ is a mod~$2$ characteristic map over $L'(J_2)$.
By the definition of the join, a subset $\sigma$ of the vertex set of $K$ is facet of $K$ if and only if $\sigma = \{v\} \cup \tau$ for a vertex $v$ of $\partial I$ and a facet $\tau$ of $L'(J_2)$.
By Theorem~\ref{thm: r3}, there is a lift $\widetilde{\mu^\R}$ of $\mu^\R$.
Then \begin{equation*}
    \begin{bmatrix}  1 & 1  & \mathds{O} \\
                                    \mathds{O} & \ast & \widetilde{\mu^\R} \end{bmatrix}
      \end{equation*}
is a lift of $\lambda^\R$.

In the rest of the section, let us focus only on IDCMs.

\begin{theorem} \label{theorem: 168}
    Let $K$ be a PL~sphere with less than $168$ facets.
    Then every injective mod~$2$ dual characteristic map over $K$ has a lift as a mod~$2$ characteristic map over $\overline{K}$.
\end{theorem}

\begin{proof}    
    Let $A$ be a $4 \times 15$ matrix over $\Z$ consisting of only $0, 1$ entries with neither repeated columns nor the zero column.
    Consider the binary matroid $\cM$ representing the mod~$2$ linear independence relations between the columns of $A$.
    More explicitly, it is the simplicial complex whose facets are the set of the indices of $4$ columns with an odd determinant in $\Z$.
    Through direct computation, we know that it has $|GL(4,\Z_2)|/4! = 840$ facets, where $GL(4,\Z_2)$ is the general linear group of degree $4$ over $\Z_2$.
    Among them, $835$ correspond to sets of $4$ vectors with determinants $\pm 1$, and the other $5$ correspond to the following sets of $4$ vectors with determinants $\pm 3$:
    \begin{center}
    	$A_1 = \left\{\begin{bmatrix}
    		1 \\
    		1 \\
    		0 \\
    		0
    	\end{bmatrix},
    	\begin{bmatrix}
    		1 \\
    		0 \\
    		1 \\
    		0
    	\end{bmatrix},
    	\begin{bmatrix}
    		1 \\
    		0 \\
    		0 \\
    		1
    	\end{bmatrix},
    	\begin{bmatrix}
    		0 \\
    		1 \\
    		1 \\
    		1
    	\end{bmatrix}\right\}$, 
    	$A_2 = \left\{\begin{bmatrix}
    	1 \\
    	1 \\
    	0 \\
    	0
    	\end{bmatrix},
    	\begin{bmatrix}
    	0 \\
    	1 \\
    	1 \\
    	0
    	\end{bmatrix},
    	\begin{bmatrix}
    	0 \\
    	1 \\
    	0 \\
    	1
    	\end{bmatrix},
    	\begin{bmatrix}
    	1 \\
    	0 \\
    	1 \\
    	1
    	\end{bmatrix}\right\}$, 
    	$A_3 = \left\{\begin{bmatrix}
    		1 \\
    		0 \\
    		1 \\
    		0
    	\end{bmatrix},
    	\begin{bmatrix}
    		0 \\
    		1 \\
    		1 \\
    		0
    	\end{bmatrix},
    	\begin{bmatrix}
    		0 \\
    		0 \\
    		1 \\
    		1
    	\end{bmatrix},
    	\begin{bmatrix}
    		1 \\
    		1 \\
    		0 \\
    		1
    	\end{bmatrix}\right\}$,
    	\medskip
    	
    	$A_4 = \left\{\begin{bmatrix}
    		1 \\
    		0 \\
    		0 \\
    		1
    	\end{bmatrix},
    	\begin{bmatrix}
    		0 \\
    		1 \\
    		0 \\
    		1
    	\end{bmatrix},
    	\begin{bmatrix}
    		0 \\
    		0 \\
    		1 \\
    		1
    	\end{bmatrix},
    	\begin{bmatrix}
    		1 \\
    		1 \\
    		1 \\
    		0
    	\end{bmatrix}\right\}$, and 
    	$A_5 = \left\{\begin{bmatrix}
    		1 \\
    		1 \\
    		1 \\
    		0
    	\end{bmatrix},
    	\begin{bmatrix}
    		1 \\
    		1 \\
    		0 \\
    		1
    	\end{bmatrix},
    	\begin{bmatrix}
    		1 \\
    		0 \\
    		1 \\
    		1
    	\end{bmatrix},
    	\begin{bmatrix}
    		0 \\
    		1 \\
    		1 \\
    		1
    	\end{bmatrix}\right\}$.
    \end{center}
    Let $a_1 = \begin{bmatrix}
    	1 \\
    	1 \\
    	0 \\
    	0
    \end{bmatrix}$, 
    $a_2 = \begin{bmatrix}
    	1 \\
    	0 \\
    	1 \\
    	0
    \end{bmatrix}$, 
    $a_3 = \begin{bmatrix}
    	1 \\
    	0 \\
    	0 \\
    	1
    \end{bmatrix}$, and 
    $a_4 = \begin{bmatrix}
    	0 \\
    	1 \\
    	1 \\
    	1
    \end{bmatrix}$
    be the four vectors in $A_1$.
    In mod~$2$, we can observe that \begin{align*}
    	A_2 &= \{a_1, a_1+a_2, a_1+a_3, a_1+a_4 \}, \\
    	A_3 &= \{a_1 + a_2, a_2, a_2+a_3, a_2+a_4 \}, \\
    	A_4 &= \{a_1 + a_3, a_2 + a_3, a_3, a_3+a_4\}, \text{ and } \\
    	A_5 &= \{a_1 + a_4, a_2 + a_4, a_3+a_4, a_4 \}.
    \end{align*}
    Notice that this combinatorial structure does not depend on the choice of $A_i$'s and $a_j$'s.

    Define $ind_A(A_i) = $ the set of indices of vectors in $A_i$ in $A$.
    For an element $g$ of $GL(4, \Z_2)$, $gA$ is obtained by a column permutation of $A$, so the five sets appear again with some other column indices.
    Suppose that $\{ind_A(A_1), ind_A(A_1), \ldots, ind_A(A_5)\} \cap \{ind_{gA}(A_1), ind_{gA}(A_2), \ldots, ind_{gA}(A_5)\} \not= \varnothing$, that is they contain a common element $ind_A(A_i) = ind_{gA}(A_j)$.
    This means that $g(A_i) = A_j$.
    Then by the property we discussed above, $g$ maps the collection of the five sets $A_i$ on itself.
    Hence for any $g \in GL(4, \Z_2)$, there are only two possibilities:
    \begin{itemize}
    	\item $\{ind_A(A_1), ind_A(A_2), \ldots, ind_A(A_5)\} \cap \{ind_{gA}(A_1), ind_{gA}(A_2), \ldots, ind_{gA}(A_5)\} = \varnothing$ or
    	\item $\{ind_A(A_1), ind_A(A_2), \ldots, ind_A(A_5)\} = \{ind_{gA}(A_1), ind_{gA}(A_2), \ldots, ind_{gA}(A_5)\}$.
    \end{itemize}
    Since every subset of column vectors of $A$ consisting of $4$ vectors with determinant $3$ or $-3$ can be transformed into $A_1$ in mod~$2$ by multiplying with a suitable $g \in GL(4, \Z_2)$, this yields a partition $\cA = \{\{ind_{gA}(A_1), ind_{gA}(A_2), \ldots, ind_{gA}(A_5)\} \mid g \in GL(4, \Z_2) \}$ of the set of facets of $\cM$ with $|\cA| = 840/5 =168$.
    In this partition, only one among the $168$ sets contains the set of vectors of determinant $\pm 3$ in $\Z$.

Now, let $\overline{\lambda^\R}$ be an IDCM over $K$.
    Then, there is an embedding of $\overline{K}$ in $\cM$ according to the index of $\overline{\lambda^\R}(i)$ in $A$ for each vertex $i$.
    If $\overline{K}$ has less than $168$ facets, then there exists an element of $\cA$ which does not intersect the set of the facets of the embedding of $\overline{K}$ by the reverse pigeonhole principle.
    This means that there exists $g \in GL(4, \Z_2)$ such that for any facet $\sigma$ of $\overline{K}$, the determinant of the matrix consisting of the $4$ vectors in $\overline{\lambda^\R} g(\sigma)$ is $1$ or $-1$ when we see the matrix as an integer $\{0, 1\}$-matrix.
    Hence, it provides a lift, as desired.
\end{proof}

Let $K$ be a seed of Picard number~$4$.
In turn, let us consider the case where $K$ is an $(n-1)$-dimensional PL~sphere that has more than or exactly $168$ facets and that supports an IDCM.
By using the list of colorable seeds of Picard number~$4$ (as in Theorem~\ref{theorem: classification}), one can see that if $n<10$, then any colorable PL~sphere has less than $168$ facets, and, hence, $n \geq 10$.
On the other hand, the condition of supporting an IDCM implies that $m \leq 15$, so $n\leq 11$.
In addition, we can check whether there is an IDCM over a given $K$ with $10 \leq n \leq 11$ by the Garrison-Scott algorithm \cite{Garrison-Scott2003} for finding all mod~$2$ characteristic maps over $K(J)$ or a modified algorithm, which is faster with small Picard numbers, for finding only IDCMs introduced in \cite{ChoiVal2021-puzzlealgo}.

\begin{algorithm}~\\ \label{algorithm:168}
      \begin{itemize}
      \item \textbf{Input:} a seed $K$.
      \item \textbf{Initialization:} \begin{itemize}
                                       \item[] $\cJ \leftarrow \{(1, \dots, 1)\}$,
                                       \item[] $\cJ_1 \leftarrow \varnothing$,
                                       \item[] $\cJ_2 \leftarrow \varnothing$,
                                       \item[] $\cR \leftarrow \varnothing$.
                                     \end{itemize} 
      \item \textbf{Output:} the list of $J$ such that $K(J)$ admits an IDCM and has $\geq 168$ facets.
      \item \textbf{Procedure:} \begin{enumerate}
                                  \item Set $\cJ \leftarrow$ the first element $\cJ[1]$ of $\cJ$, and remove it from $\cJ$.
                                  \item If there is no IDCM over $K(J)$, then add $J$ to $\cJ_1$, and go to (1).
                                  \item Add $J$ to $\cJ_2$.
                                  \item If $K(J)$ has  $\geq 168$ facets, then add $J$ to $\cR$.
                                  \item If $\cJ \not= \varnothing$, then go to (1).
                                  \item If $\dim(K(J)) = 10$, then return $\cR$.
                                  \item For $i=1, 2$, set $\cJ'_i \leftarrow$ the set of $J'$ such that $J'$ is equal to $J$ for some $J$ in $\cJ_i$ except $k$th component for some $k$, and $k$th component of $J'$ is 1 larger than the one of $J$.
                                  \item Set $\cJ_1 \leftarrow \cJ'_1$, $\cJ_2 \leftarrow \varnothing$, and $\cJ \leftarrow \cJ'_2 \setminus \cJ'_1$, and go to (1).
                                \end{enumerate}
    \end{itemize}
\end{algorithm}

After all these refining by Algorithm~\ref{algorithm:168}, there only remain twenty-one PL~spheres that have more than or exactly $168$ facets and ninety-six IDCMs over them.
Table~\ref{table:IDCM_over_168} shows that the number of such PL~spheres and IDCMs over them.
    \begin{table}[]
		\begin{tabular}{lrrrrrr}
		\toprule
		$n$   &        & 10  &  11  \tabularnewline \midrule
		\multicolumn{2}{l}{$K$}   & 11 & 10 \tabularnewline  \rule{0pt}{3ex}
		&seeds                & 11 & 4 \tabularnewline
		&non-seeds            & 0  & 6 \tabularnewline \midrule
		\multicolumn{2}{l}{$(K, \overline{\lambda^\R})$}  & 11 & 85 \tabularnewline \rule{0pt}{3ex}
		&seeds            & 11 & 5  \tabularnewline
		&non-seeds             & 0  & 80 \tabularnewline
		\bottomrule
	   \end{tabular}
	   \medskip
    \caption{The numbers of PL~sphere $K$ supporting an IDCM and having $\geq 168$ facets (above) and the total number of IDCMs over them (below).} \label{table:IDCM_over_168}
    \end{table}

Two IDCMs $\Lambda_1^\R$ and $\Lambda_2^\R$ over a non-seed $K$ are said to be \emph{symmetric} if they can be expressed by $\Lambda_1^\R = \lambda_1^\R \wedge_v \lambda_2^\R$ and $\Lambda_2^\R = \lambda_2^\R \wedge_v \lambda_1^\R$ for some IDCMs $\lambda_1^\R$ and $\lambda_2^\R$.
Although they are distinguished as IDCMs, the existence of their lifts are equivalent. 
Therefore, it is enough to consider all IDCMs up to symmetry.

Reducing symmetries, there is only one IDCM over each non-seed $K$ of Table~\ref{table:IDCM_over_168}.
See Table~\ref{table:IDCM_up_to_symmetry}.

\begin{table}[]
		\begin{tabular}{lrrrrrr}
		\toprule
		$n$   &        & 10  &  11  \tabularnewline \midrule
		\multicolumn{2}{l}{$(K, \overline{\lambda^\R})$}  & 11 & 11 \tabularnewline \rule{0pt}{3ex}
		&seeds            & 11 & 5  \tabularnewline
		&non-seeds             & 0  & 6 \tabularnewline
		\bottomrule
	\end{tabular}
	\medskip
    \caption{The number of IDCMs up to symmetry over the PL~spheres $K$ having $\geq 168$ facets.} \label{table:IDCM_up_to_symmetry}
\end{table}

The final step is to check whether all twenty-two pairs $(K, \lambda^\R)$ have $\{0, \pm 1\}$-lifts by the following simple algorithm.
\newpage

\begin{algorithm}~\\
    \begin{itemize}
      \item \textbf{Input:} the cofacets $CF$ of $K$ and a mod~$2$ dual characteristic map $\overline{\lambda^\R}$ over $K$
      \item \textbf{Initialization:} \begin{itemize}
                                       \item[] $I \leftarrow$ the list of the indices of nonzero entries in $\lambda^\R$,
                                       \item[] $i \leftarrow 0$.
                                     \end{itemize} 
      \item \textbf{Output:} a $\{0, \pm 1\}$-lift of $\overline{\lambda^\R}$ if it exists and $0$ otherwise.
      \item \textbf{Procedure:} \begin{enumerate}
                                  \item If $i = |I|$, then return $0$.
                                  \item Set $S \leftarrow$ the list of all $i$-subsets of $I$.
                                  \item If $S = \varnothing$, then set $i \leftarrow i+1$, and go to (1).
                                  \item Set $s \leftarrow$ the first element $S[1]$ of $S$, and remove it from $S$.
                                  \item Replace $1$'s in $\overline{\lambda^\R}$ with indices in $s$ by $-1$'s.
                                  \item If there is a cofacet $cf \in CF$ such that the determinant of the matrix consisting of the rows of $\overline{\lambda^\R}$ corresponding to $cf$ is not $\pm 1$, then go to (3).
                                  \item Return $\overline{\lambda^\R}$.
                                \end{enumerate}
    \end{itemize}
\end{algorithm}

In conclusion, we have the following theorem.

\begin{theorem}
    For a PL~sphere of Picard number~$4$, any quasi-injective mod~$2$ dual characteristic map over $K$ has a lift as a mod~$2$ characteristic map over $\overline{K}$.
\end{theorem}

The database containing the PL~spheres of Picard number~$4$ admitting an IDCM are available on the second author's Github repository:
\begin{center}
\url{https://github.com/Hyeontae1112/TWI}
\end{center}

\bibliographystyle{amsplain}
\bibliography{reference2024}
\end{document}